\documentclass[11pt,reqno]{amsart}
\usepackage[T1]{fontenc}
\usepackage[utf8]{inputenc}
\usepackage{lmodern}
\usepackage[letterpaper,margin=1in]{geometry}
\usepackage{amsmath,amssymb,mathtools}
\usepackage{microtype}
\usepackage{hyperref,comment}
\usepackage{xcolor}
\usepackage{tikz}
\usepackage{pgfplots}
     \pgfplotsset{compat=1.18}
\hypersetup{hidelinks,pdftitle={Sharp derivative inequalities for uncentered maximal functions},pdfauthor={Jose Madrid},pdfsubject={Sharp derivative inequalities, Riesz variation, and classical endpoints on the real line}}
\input{glyphtounicode}
\newtheorem{theorem}{Theorem}
\newtheorem{proposition}[theorem]{Proposition}
\newtheorem{lemma}[theorem]{Lemma}

\theoremstyle{remark}
\newtheorem{remark}[theorem]{Remark}
\numberwithin{equation}{section}
\newcommand{\R}{\mathbb R}
\newcommand{\M}{\widetilde M}
\newcommand{\norm}[1]{\left\lVert#1\right\rVert}
\DeclareMathOperator{\Lip}{Lip}
\DeclareMathOperator{\Var}{Var}
\DeclareMathOperator{\supp}{supp}
\newcommand{\AC}{\mathrm{AC}}

\allowdisplaybreaks[1]
\title[Sharp derivative inequalities for maximal functions]{Sharp derivative inequalities for uncentered maximal functions}

\begin{document}
\author[Sung-Yi Liao]{Sung-Yi Liao}
\address{Department of Mathematics, Virginia Polytechnic Institute and State University, 225 Stanger Street, Blacksburg, VA 24061-1026, USA}
\email{sungyi@vt.edu}
\author[Jos\'e Madrid]{Jos\'e Madrid}
\address{Department of Mathematics, Virginia Polytechnic Institute and State University, 225 Stanger Street, Blacksburg, VA 24061-1026, USA}
\email{josemadrid@vt.edu}
\author[Eyvindur Ari Palsson]{Eyvindur Ari Palsson}
\address{Department of Mathematics, Virginia Polytechnic Institute and State University, 225 Stanger Street, Blacksburg, VA 24061-1026, USA}
\email{palsson@vt.edu}

\begin{abstract}
We determine the optimal constant in the $L^p$-derivative inequality for the uncentered Hardy--Littlewood maximal operator on the real line for every $1<p<\infty$. It is the unique number $\gamma_p\in [0,1]$ satisfying $(2p-1)\gamma_p+(p-1)\gamma_p^{(2p-1)/(p-1)}=p$. As $p$ goes from $1$ to $+\infty$, $\gamma_p$ decreases from $1$ to $\sqrt2-1$, and at the endpoints the limiting arguments recover the classical total-variation and Lipschitz inequalities, with the qualitative regularity ingredient at the $BV$ endpoint supplied by Aldaz and P\'erez L\'azaro.
\end{abstract}
\maketitle

\section{Introduction}

We study the extent to which the uncentered Hardy--Littlewood maximal operator decreases derivative seminorms on the real line. Our main result determines the optimal constant, which is characterized by a scalar equation, in the $L^p$-derivative inequality for every $1<p<\infty$. Additionally, an explicit family of even, compactly supported functions witnesses its sharpness. The result connects the sharp total-variation inequality of Aldaz and P\'erez L\'azaro~\cite{AP} with the sharp Lipschitz inequality of Aldaz, Colzani, and P\'erez L\'azaro~\cite{ACP}. 

\subsection{Background and notation}
For $f\in L^1_{\mathrm{loc}}(\R)$, define the uncentered Hardy--Littlewood maximal operator by
\begin{equation}\label{eq:defM}
 \M f(x):=
 \sup_{\substack{r,s\ge0\\r+s>0}}
 \frac{1}{r+s}\int_{x-r}^{x+s}|f(t)|\,dt.
\end{equation}
The supremum is over all finite intervals containing $x$. All norms and integrals without a specified domain are over $\R$.

Kinnunen~\cite{Kinnunen} initiated the Sobolev regularity theory of maximal functions by proving boundedness on $W^{1,p}(\R^d)$ for $1<p\leq\infty$. His proof establishes pointwise control of the weak derivatives. Although his original formulation uses the centered operator, the same argument applies to the uncentered operator. In one dimension it gives
\[
 |(\M f)'(x)|\le\M(f')(x)
 \quad\text{for almost every }x\in\R,
 \qquad f\in W^{1,p}(\R).
\]
Here $\M(f')=\M(|f'|)$ by definition. As a corollary, taking $L^p$ norms yields
\begin{equation}\label{eq:classical-derivative}
 \norm{(\M f)'}_p\le\norm{\M(f')}_p
 \le C_p\norm{f'}_p,
\end{equation}
where $C_p=\norm{\M}_{L^p(\R)\to L^p(\R)}$. Grafakos and Montgomery-Smith~\cite{GMS} determined this operator norm $C_p$: it is the unique positive solution of
\begin{equation}\label{eq:operator_norm}
(p-1)x^p-px^{p-1}-1=0.
\end{equation}
In particular, $(p-1)x=p+x^{1-p}$, so $C_p>p/(p-1)$ and $C_p\to\infty$ as $p\downarrow1$. Besides the boundedness, the map $f\mapsto (\M f)'$, in the weak derivative sense, is continuous from $W^{1,p}(\R)\to L^p(\R)$ for $1\leq p<\infty$ \cite{CMP,L}.

Nevertheless, the bound~\eqref{eq:classical-derivative} provides only an upper bound and does not capture the derivative's sharp behavior. In particular, as $p\downarrow1$, $C_p\to\infty$. On the other hand, Aldaz and P\'erez L\'azaro~\cite{AP} proved that the maximal function of a bounded-variation function is absolutely continuous and has no greater total variation. In particular, for $f\in W^{1,1}(\mathbb{R})$,
\begin{equation}\label{eq:derivative_bound}
\norm{(\M f)'}_1\leq\norm{f'}_1.
\end{equation}
%Also, in $W^{1,1}$ setting, Liu, Chen, and Wu \cite{LCW} use a different approach to obtain the same result. 
Thus, the classical $L^p$ operator norm diverges as $p\downarrow1$, whereas the sharp derivative bound at $p=1$ has a constant of $1$. 

Theorem~\ref{thm:main} provides a better bound than the $C_p$ bound for the first derivative case, and determines the optimal constants: it replaces $C_p>1$ with $\gamma_p<1$ for every $p>1$, with $\gamma_p\to1$ as $p\downarrow1$. At the other endpoint $p=\infty$, Aldaz, Colzani, and P\'erez L\'azaro~\cite{ACP} proved the sharp Lipschitz inequality with constant $\sqrt{2}- 1$, and our result also provides a different approach to recover these endpoint inequalities, respectively.

Grafakos and Kinnunen~\cite[Theorem 2.1 and Corollary 2.3]{GK} extended the sharp $L^p$ bound result \eqref{eq:operator_norm}, together with a refined weak-type $(1,1)$ estimate, to nonnegative Borel measures $\mu$ on $\R$ without a doubling assumption on the Borel measure $\mu$. Their constants are independent of the measure and sharp for Lebesgue measure. 

As for the discrete setting, for a function $f:\mathbb{Z}\to\mathbb{R}$, define 
\[\M_\mathbb{Z} f(n):=\sup_{\substack{r,s\geq0}}\frac{1}{r+s+1}\sum_{i=-r}^s|f(n+i)|,\qquad f'(n):=f(n+1)-f(n).\]
Then, the corresponding derivative map $f\mapsto (\M_\mathbb{Z} f)'$ is bounded from $W^{1,p}(\mathbb{Z})$ to $\ell^p(\mathbb{Z})$ for $1\leq p\leq\infty$. Here, $W^{1,p}(\mathbb{Z})$ is the Sobolev space. Bober, Carneiro, Hughes, and Pierce \cite[Theorem 1]{BCHP} obtained the discrete version of the sharp first derivative estimate for the case $p=1$:
\begin{equation}\label{eq:l1Z_bound}
\norm{(\M f)'}_{\ell^1(\mathbb{Z})}\leq\norm{f'}_{\ell^1(\mathbb{Z})}.
\end{equation}
For the case $p=\infty$, Gonz\'alez-Riquelme and the second author \cite[Theorem 20]{GM} proved the discrete version of the sharp first derivative estimate:
\begin{equation}\label{eq:linftyZ_bound}
\norm{(\M f)'}_{\ell^\infty(\mathbb{Z})}\leq\frac{1}{2}\norm{f'}_{\ell^\infty(\mathbb{Z})}.
\end{equation}
For $1<p<\infty$, Weigt and the authors \cite[Theorems 2.3 and 2.4]{LMPW} determined the discrete version of the sharp first-difference constant for characteristic functions and showed that the constant for general functions is strictly larger. 

It is noteworthy that for $p=1$, the optimal constant is the same in both the discrete setting \eqref{eq:l1Z_bound} and the continuous setting \eqref{eq:derivative_bound}. However, on the other endpoint, for $p=\infty$, the optimal constant in the estimate \eqref{eq:endpoint-Lip} is $\sqrt{2}-1$ in the continuous setting while the constant is $\frac{1}{2}$ in the discrete setting \eqref{eq:linftyZ_bound}. Nevertheless, given a compact interval $I\subseteq\mathbb{R}$, one may consider the interval restricted maximal operator 
\[\M_I f(x):=\sup_{\substack{x\in J\subset I\\0<|J|<\infty}}\frac{1}{|J|}\int_J |f|,\]
which will share the same optimal constant with the discrete setting at $p=\infty$, see \cite[Corollary 4.5]{ACP}. Therefore, the corresponding estimate between the integer and the real number is not always straightforward. 

Besides the domains $\mathbb{R}$ and $\mathbb{Z}$, one may also consider the higher dimensional case, like $\mathbb{R}^n,\mathbb{Z}^n$, and proper subsets of $\mathbb{R}^n$ \cite{C,CM,HO}. Additionally, beyond the uncentered maximal operator, we may study a similar question for the centered maximal operator, non-tangential maximal operator, and the maximal operator of convolution type \cite{C,CS,HO,R}. 

The arguments and results below all concern derivative seminorms for the continuous uncentered maximal operator.

For $1\le p<\infty$ we use the Riesz $p$-variation
\begin{equation}\label{eq:riesz}
 \Var_p(f):=\sup_{x_0<\cdots<x_m}
 \left(\sum_{j=1}^m
 \frac{|f(x_j)-f(x_{j-1})|^p}{|x_j-x_{j-1}|^{p-1}}
 \right)^{1/p}.
\end{equation}
We write $\Var(f)=\Var_1(f)$. If $f$ is locally absolutely continuous and $f'\in L^p$, then $\Var_p(f)=\norm{f'}_p$; conversely, finite Riesz $p$-variation for $p>1$ implies this regularity and identity; see~\cite[Chapter IX \textsection 4, Theorem 7]{N}. For a bounded-variation function $f$, $Df$ denotes the distributional derivative and $|Df|(\R)$ the mass of its total variation measure. We use the notation $BV$ here in the homogeneous sense. In other words, $f$ is $BV$ means $f\in L^1_{\mathrm{loc}}(\R)$ and $|Df|(\R)<\infty$, without requiring $f\in L^1(\R)$. Notably, $|Df|(\R)$ is independent of changes on null sets, whereas $\Var(f)$ is not. For example, a characteristic function $\chi_{\{0\}}$ has distributional variation zero but pointwise variation two. However, they will agree for the right-continuous representative $x\mapsto f(x+)$; see~\cite[Lemma 3.7]{AP} for the corresponding identification using a canonical representative. 

\subsection{The sharp derivative and Riesz variation inequality}
For $1<p<\infty$, let $\gamma_p\in [0,1]$ be the unique solution of
\begin{equation}\label{eq:gamma}
 (2p-1)x+(p-1)x^{(2p-1)/(p-1)}=p\qquad\text{for}\qquad x\in [0,1].
\end{equation}
The left side is strictly increasing in $x$, so the intermediate value theorem immediately yields the existence and uniqueness of $\gamma_p$. Let $\AC_{\mathrm{loc}}(\mathbb{R})$ be the space of locally absolutely continuous functions on $\mathbb{R}$.

\begin{theorem}\label{thm:main}
Let $1<p<\infty$. Suppose that $f\in\AC_{\mathrm{loc}}(\R)$, $f'\in L^p(\R)$, and that $\M f$ is finite at some point. Then $\M f$ is finite everywhere, locally absolutely continuous, and
\begin{equation}\label{eq:main}
 \norm{(\M f)'}_p=\Var_p(\M f)
 \le\gamma_p\norm{f'}_p=\gamma_p\Var_p(f).
\end{equation}
In particular, the inequality holds for every $f\in W^{1,p}(\R)$, which implies $\M f\in W^{1,p}(\R)$. The constant $\gamma_p$ is optimal and there is a sequence of nonnegative and even functions $f_n$ supported in $[-1,1]$ and nonincreasing on $[0,1)$ such that 
\[\lim_{n\to\infty}\frac{\norm{(\M f_n)'}_p}{\norm{f_n'}_p}=\gamma_p.\]
\end{theorem}
\begin{remark}
Specifically, for $p=2$, we have
\[\gamma_2=\sqrt[3]{\sqrt2+1}-\sqrt[3]{\sqrt2-1}\approx0.5961,
 \qquad \gamma_2^3+3\gamma_2=2.\]

Additionally, in order to obtain a rough upper bound of $\gamma_p$, one may ignore the second term $(p-1)x^{(2p-1)/(p-1)}$ in \eqref{eq:gamma} and obtain the estimate
\[\gamma_p\leq\frac{p}{2p-1}.\]

Compared to the argument in \eqref{eq:classical-derivative}, which gives a bound $C_2=\sqrt{2}+1\approx2.4142$ for $p=2$, our result indeed improves the classical bound.
\end{remark}

Note that every $W^{1,p}(\R)$ function has a bounded, locally absolutely continuous representative, so the finiteness hypothesis is automatic; the classical $L^p$ maximal inequality also gives $\M f\in L^p$. The one-point finiteness assumption is needed in the more general homogeneous formulation: For example, consider $f(x)=\log(2+|x|)$. Then, we have $f\in\AC_{\mathrm{loc}}(\R)$ and $f'\in L^p$ for all $p>1$, but $\M f\equiv\infty$. 

\subsection{The optimal constant and the classical endpoints}
The constant in Theorem~\ref{thm:main} has an exact relation to the sharp H\"older constants of Aldaz, Colzani, and P\'erez L\'azaro~\cite[Remark 3.3, formula (3.3.1)]{ACP}. For $0<\alpha\le1$, put
\begin{equation}\label{eq:holder-constant}
 H_\alpha:=\min_{0\le t\le1}
 \frac{1+t^{1+\alpha}}{(1+\alpha)(1+t)}.
\end{equation}
Their theorem identifies $H_\alpha$ as the optimal constant in
\begin{equation}\label{eq:Holder_norm}
[\M f]_{C^{0,\alpha}}\le H_\alpha[f]_{C^{0,\alpha}},
\end{equation}
for all globally $\alpha$-H\"{o}lder functions $f$ with finite maximal function, where
\[
 [f]_{C^{0,\alpha}}:=\sup_{x\ne y}
 \frac{|f(x)-f(y)|}{|x-y|^\alpha},
 \quad\text{ especially }\quad\Lip(f)=[f]_{C^{0,1}}.
\]
The following proposition demonstrates the relation between $H_\alpha$ and $\gamma_p$ and the limit properties:
\begin{proposition}\label{prop:coefficient}
For every $1<p<\infty$,
\begin{equation}\label{eq:critical-holder}
 \gamma_p=H_{1-1/p}.
\end{equation}
The function $p\mapsto\gamma_p$ is strictly decreasing, with
\begin{equation}\label{eq:limits}
 \lim_{p\downarrow1}\gamma_p=1,
 \qquad \lim_{p\to\infty}\gamma_p=\sqrt2-1.
\end{equation}
If $r=\sqrt2-1$, then
\begin{equation}\label{eq:asymptotic}
 \gamma_p=r+\frac{1-r-r^2\log r}{2(1+r)}\,\frac1p+O(p^{-2}).
\end{equation}
\end{proposition}

The scalar minimum $H_\alpha$ and \eqref{eq:Holder_norm} are due to~\cite{ACP}. Our new assertion is that this same number, at $\alpha=1-1/p$, is the exact coefficient for the $L^p$ norm of the derivative. We recall the classical endpoint results below as background, and recover their quantitative bounds in Section~\ref{sec:endpoints} by taking the limit.

For every $f\in L^1_{\mathrm{loc}}(\R)$ with $|Df|(\mathbb{R})<\infty$, there is a bounded-variation representative that coincides with $f$ almost everywhere. In this case, Aldaz and P\'erez L\'azaro~\cite[Theorem 2.5]{AP} proved that $\M f$ is absolutely continuous, and that
\begin{equation}\label{eq:endpoint-BV}
 \norm{(\M f)'}_1=\Var(\M f)\le|Df|(\R).
\end{equation}
In particular, the right side is $\norm{f'}_1$ when $f\in W^{1,1}(\R)$. For the other endpoint case, Aldaz, Colzani, and P\'erez L\'azaro~\cite[Corollary 3.6]{ACP} established the sharp whole-line Lipschitz bound
\begin{equation}\label{eq:endpoint-Lip}
 \norm{(\M f)'}_\infty\le(\sqrt2-1)\norm{f'}_\infty
\end{equation}
for globally Lipschitz inputs with finite maximal function. Both results are classical and sharp. Section~\ref{sec:endpoints} explains how their quantitative inequalities follow from Theorem~\ref{thm:main} by limiting arguments and approximation. 

\subsection{Proof strategy}
We first prove Theorem \ref{thm:main} for nonnegative, continuous, compactly supported, and piecewise-affine functions. On the smooth intervals where $(\M f)'$ is nonzero, the maximizing interval is unique and one may verify that the evaluation point $x$ will be one endpoint. We write $a(x)$ for the other endpoint. Local differentiation identities, adapted from Weigt \cite[Lemma 2.2]{W}, relate $(\M f)'$ to $f$ and to the motion of $a(x)$.

The key observation is that the change of variables $x\mapsto a(x)$ counts each part of $\int_\mathbb{R}{|f'|}^p$ at most once. Combining this fact with the local identities, integration by parts, and H\"{o}lder’s inequality yields an estimate relating $\norm{(\M f)'}_p$ to $\norm{f'}_p$. After normalization, this estimate gives the scalar inequality associated with \eqref{eq:gamma}.

\begin{comment}
To prove Theorem \ref{thm:main}, the central issue is to estimate the integrals of certain products of $(\M f)'$ and $|f'|$ while keeping track of which parts of $f$ contribute to the estimate. We begin with nonnegative, compactly supported, and piecewise-affine functions. On each interval where $(\M f)'$ is nonzero, the maximizing interval is unique and has the evaluation point $x$ as one endpoint. We denote the other endpoint by $a(x)$. The local differentiation identities in Weigt~\cite[Lemma 2.2]{Weigt} relate $(\M f)'$ to $f$ and to the motion of $a(x)$. 

The key additional observation is that changing variables from $x$ to $a(x)$ counts each part of the integral of $|f'|^p$ at most once. Combining this observation with the local differentiation identities and integration by parts gives an inequality relating $\|(\M f)'\|_p$ to $\|f'\|_p$. After normalization, this becomes the desired inequality that determines the sharp constant.
\end{comment}
Sections \ref{sec:maximizing_interval} and \ref{sec:smooth_pieces} establish the local formulas, the desired partition, and the injectivity of the endpoint map. Section \ref{sec:proof:piecewise_affine} proves the derivative estimate for piecewise-affine inputs. Section \ref{sec:approximation} extends it by approximation to the full class of inputs in Theorem \ref{thm:main}, and Section \ref{sec:sharpness} establishes sharpness through an explicit family of functions. Finally, Section \ref{sec:endpoints} proves Proposition \ref{prop:coefficient} and derives the quantitative endpoint bounds.

%Section~\ref{sec:ordinary} carries out this argument. To begin with, Sections \ref{sec:maximizing_interval} and \ref{sec:smooth_pieces} start with some introduction and observations of $a(x)$ and $(\M f)'$. Section \ref{sec:proof:piecewise_affine} estimates the integrals of certain products of $(\M f)'$ and $|f'|$ via injectivity, non-negativity, and previous Lemmas. Then, with H\"{o}lder's inequality and normalization, the integral becomes the desired estimate of the ratio $\norm{(\M f)'}_p/\norm{f'}_p$ when $f$ is non-negative, compactly supported, and piecewise-affine. In Section \ref{sec:approximation}, we use an approximation argument to complete the proof of \eqref{eq:main} in Theorem \ref{thm:main} for functions satisfying the conditions in Section \ref{sec:proof:piecewise_affine}. By directly constructing a sequence of functions such that the ratio $\norm{(\M f)'}_p/\norm{f'}_p$ approaches the constant in \eqref{eq:main}, Section \ref{sec:sharpness} finishes the proof of Theorem \ref{thm:main}. Next, Section~\ref{sec:endpoints} proves Proposition \ref{prop:coefficient} and explains how limiting arguments recover the classical endpoint bounds.

\section{Proof of Theorem \ref{thm:main}}\label{sec:ordinary}

We first assume that $f\ge0$ is continuous, compactly supported, and affine on each of finitely many intervals. All differentiation and integration by parts will be performed on finitely many smooth pieces. The passage to general inputs and the sharpness argument conclude this section.

\subsection{Maximizing intervals}\label{sec:maximizing_interval}
By sublinearity, we have
\begin{equation}\label{eq:basic-Lip}
 |\M f(x+h)-\M f(x)|\le\Lip(f)|h|.
\end{equation}
This Lipschitz bound will provide the regularity needed below.

On the open set
\[
 \Omega:=\{x\in\R:\M f(x)>f(x)\},
\]
the supremum defining $\M f(x)$ is attained by a finite, nondegenerate interval. Note that short intervals $I$ have average $\frac{1}{|I|}\int_If$ close to $f(x)$, while an interval of length $\ell$ has average at most $\norm{f}_1/\ell$. The remaining endpoints range over a compact set.

At almost every point of the contact set $\{x\in\mathbb{R}:\M f(x)=f(x)\}$,
\begin{equation}\label{eq:zero-contact}
 (\M f)'(x)=f'(x)=0.
\end{equation}
Indeed, since both $\M f$ and $f$ are Lipschitz continuous, ${(\M f)}'=f'$ on the contact set. Additionally, a nonzero derivative of $f$ at $x$ means a one-sided average exceeds $f(x)$ to first order, so 
\begin{equation}\label{eq:one_side_suffices}
\M f(x)=\max\left\{\frac{1}{h}\int_{[x-h,x]}f,\frac{1}{h}\int_{[x,x+h]}f\right\}>f(x)\text{ for some sufficiently small }h>0,
\end{equation}
which contradicts the assumption $\M f(x)=f(x)$.

The local formulas in the next lemma are the differentiation identities in Weigt~\cite[Lemma 2.2, equations (2.3)--(2.5)]{W}, proved there under a unimodality and two-sided slope hypothesis. We include the details and adapt his proof for completeness in the present piecewise-affine setting.

\begin{lemma}\label{lem:endpoint}
Suppose that $x\in\Omega$, that $\M f$ is $C^1$ in a neighborhood of $x$, and that $(\M f)'(x)<0$. The maximizing interval is unique and has the form $[a(x),x]$, with
\begin{equation}\label{eq:local-identities}
 f(a(x))=\M f(x),
 \qquad
 (\M f)'(x)=\frac{f(x)-\M f(x)}{x-a(x)}.
\end{equation}
Additionally, if $\M f$ is twice differentiable at $x$ and $f'(a(x))\ne0$, one also has
\begin{equation}\label{eq:weigt-identities}
 \begin{split}
 a'(x)&=\frac{(\M f)'(x)}{f'(a(x))},\\
 (\M f)''(x)&=
 \frac{f'(x)-(2-a'(x))(\M f)'(x)}{x-a(x)}.
 \end{split}
\end{equation}
When $(\M f)'(x)>0$, the maximizing interval is $[x,a(x)]$; the same formulas hold, with $x-a(x)<0$.
\end{lemma}
\begin{proof}

Suppose first that $(\widetilde M f)'(x)<0$. For any interval $[a,b]$ containing $x$, since 
\[\frac{1}{b-a}\int_a^bf\leq\max\left\{\frac{1}{x-a}\int_a^xf,\frac{1}{b-x}\int_x^bf\right\},\]
it suffices to consider those intervals such that $x$ is an endpoint when taking the supremum of $\M f$. Additionally, if the supremum is achieved by an interval in the form of $[x,b]$, then $\M f(x+h)\geq\M f(x)$ for all sufficiently small $h>0$, which contradicts the assumption $(\widetilde M f)'(x)<0$. Thus, a maximizing interval must be in the form of $[a,x]$.

Fix such an interval $[a,x]$. Varying its left endpoint and using the continuity of $f$ and the first derivative test, we obtain
\[
 0=\left.\frac{d}{du}
       \left(\frac{1}{x-u}\int_u^x f(t)\,dt\right)\right|_{u=a}
   =\frac{\widetilde M f(x)-f(a)}{x-a}.
\]
Hence $f(a)=\widetilde M f(x)$. On the other hand, the function
\[
 t\longmapsto \frac{1}{t-a}\int_a^t f(u)\,du
\]
touches the graph of $\M f$ from below at $t=x$ and is bounded by $\M f$ from above. Since both functions are differentiable there, their derivatives must agree at $t=x$ by the boundedness, giving
\[
 (\widetilde M f)'(x)
   =\left.\frac{d}{dt}\left(\frac{1}{t-a}\int_a^t f(u)\,du\right)\right|_{t=x}=\frac{f(x)-\widetilde M f(x)}{x-a}.
\]
In particular,
\begin{equation}\label{eq:formula_a}
a=x-\frac{f(x)-\widetilde M f(x)}{(\widetilde M f)'(x)},
\end{equation}
so the maximizing interval is unique. Denote its left endpoint by
$a(x)$. This completes the proof of \eqref{eq:local-identities} and also shows that $a(x)$ is differentiable when $\M f$ is twice differentiable.

Moreover, consider a point $x$ such that $f$ is $C^2$ at $x$ and $f'(a(x))\neq0$. Using the identity $f(a(x))=\M f(x)$ and differentiating both sides, we directly get 
\[f'(a(x))a'(x)=(\M f)'(x)\Rightarrow a'(x)=\frac{(\M f)'(x)}{f'(a(x))}.\]

Finally, differentiating \eqref{eq:local-identities} gives
\begin{align*}
 (\widetilde M f)''(x)
 &=\frac{f'(x)-(\widetilde M f)'(x)}{x-a(x)}
   -\frac{(f(x)-\widetilde M f(x))(1-a'(x))}{(x-a(x))^2}\\
 &=\frac{f'(x)-(2-a'(x))(\widetilde M f)'(x)}{x-a(x)},
\end{align*}
which proves \eqref{eq:weigt-identities}. Applying the argument to $t\mapsto f(-t)$
gives the case $(\widetilde M f)'(x)>0$, with maximizing interval
$[x,a(x)]$ and the same formulas, now with $x-a(x)<0$.
\end{proof}

\subsection{Smooth pieces and endpoint injectivity}\label{sec:smooth_pieces}
Under the assumption that $f\ge0$ is continuous, compactly supported, and affine on each of finitely many intervals, the following lemmas provide a way to sum the local calculations over all of $\R$ without intersection.
\begin{lemma}[Finite smooth partition]\label{lem:branches}
There is a finite partition of $\R$ into points and open intervals $J$ such that every $f|_J$ is affine, every $\M f|_J$ is smooth, and
\[
 (\M f)'|_J>0,\qquad (\M f)'|_J<0,\qquad\text{or}\qquad
 (\M f)'|_J\equiv0.
\]
The partition can also be chosen so that $(\M f)'$ is monotone on each $J$.
\end{lemma}

\begin{proof}
Write $f(t)=m_it+b_i$ on its finitely many open affine intervals $I_i$. Let $P(t)=\int_{-\infty}^tf(y)dy$ and $P_i=P|_{I_i}$. On $I_i$, we have
\[P_i(t):=\tfrac12m_it^2+b_it+c_i\]
where $c_i$ is a constant such that $P(t)$ is continuous on $\mathbb{R}$. For $x\in I_i$, set
\[
 A(u,x):=\begin{cases}
\frac{P(x)-P(u)}{x-u}&x\neq u,\\
f(x)&x=u,
 \end{cases}
\]
which is the average of $f$ between $u$ and $x$. By \eqref{eq:one_side_suffices},
\[\M f(x)=\sup_{u\ne x}A(u,x).\]
If a maximizing endpoint $u=u_0$ lies in $I_j$, differentiation gives
\begin{equation*}
\begin{split}
&0=\partial_u A(u_0,x)
 =\frac{-f(u)}{x-u}+\frac{P(x)-P(u)}{{(x-u)}^2}=\frac{A(u,x)-f(u)}{x-u},\\
\Rightarrow&0=\frac{1}{2}m_ju^2-m_jxu+P_i(x)-b_jx-c_j.
\end{split}
\end{equation*}
For $m_j\ne0$, define
\[
 D_{ij}(x):=x^2-\frac{2}{m_j}\bigl(P_i(x)-b_jx-c_j\bigr),
 \qquad u_{ij}^{\pm}(x):=x\pm\sqrt{D_{ij}(x)},
\]
where $u_{ij}^{\pm}(x)$ are two candidates of endpoint $a(x)$.
Thus the possible values of $\M f(x)$, for $x\in I_i$, only come from the following Table \ref{tab:candidates_of_Mf}:

\begin{table}[ht]
\centering
\small
\renewcommand{\arraystretch}{1.4}
\begin{tabular}{c|c|c}
Case accounted for & Value & Condition for inclusion \\ \hline
Trivial averaging interval
& $f(x)$
& Always \\

Nontrivial averaging interval with $m_j\ne0$
& $m_ju_{ij}^{\pm}(x)+b_j$
& $D_{ij}(x)\geq0,u_{ij}^{\pm}(x)\in\overline{I_j}$ \\

Nontrivial averaging interval with $m_j=0$
& $b_j$
& $P_i(x)-b_jx-c_j=0$
\end{tabular}
\caption{Candidates for $\M f(x)$, organized by the location
of the possible maximizing endpoints. These cases exhaust the
possibilities.}
\label{tab:candidates_of_Mf}
\end{table}

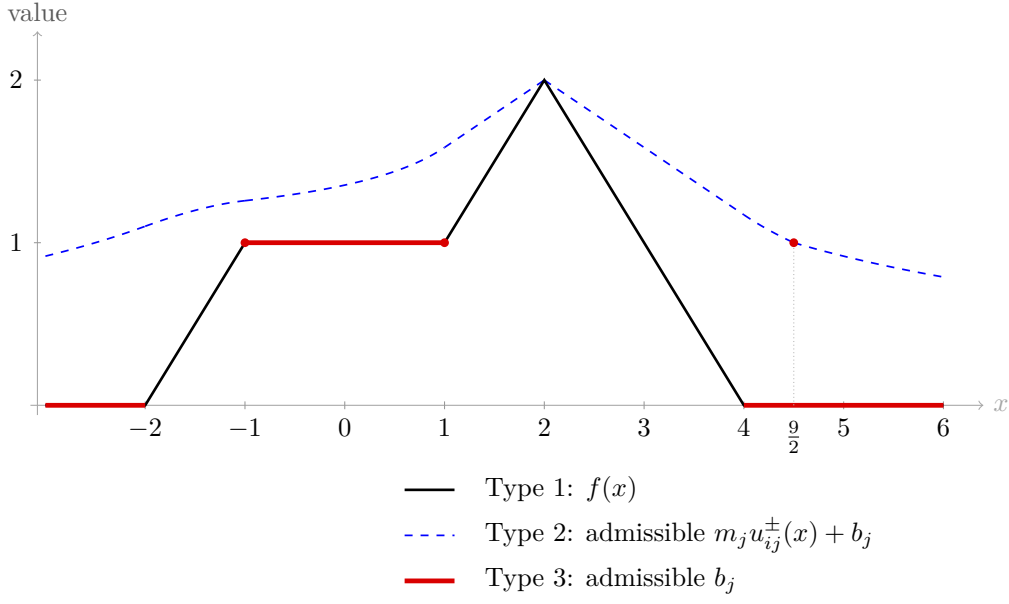
\begin{figure}[ht]
\centering
\begin{tikzpicture}[x=1.32cm,y=2.15cm,
  typeone/.style={black,line width=1pt},
  typetwo/.style={blue,dashed,line width=.65pt},
  typethree/.style={red!85!black,line width=1.8pt},
  every node/.style={font=\small}]
  \draw[->,gray!70] (-3.15,0)--(6.4,0) node[right] {$x$};
  \draw[->,gray!70] (-3.08,-.06)--(-3.08,2.3)
    node[above,text=black!65] {value};
  \foreach \x in {-2,-1,0,1,2,3,4,5,6}
    \draw[gray!80] (\x,.024)--(\x,-.024)
      node[below,text=black] {$\x$};
  \foreach \y in {1,2}
    \draw[gray!80] (-3.115,\y)--(-3.045,\y)
      node[left=3pt,text=black] {$\y$};

  % Type 2: f(u_{i4}^+(x))=4-u_{i4}^+(x), x<2.
  \draw[typetwo] plot[domain=-3:-2,samples=70]
    (\x,{4-\x-sqrt(\x*\x-8*\x+4)});
  \draw[typetwo] plot[domain=-2:-1,samples=70]
    (\x,{4-\x-sqrt(2*\x*\x-4*\x+8)});
  \draw[typetwo] plot[domain=-1:1,samples=100]
    (\x,{4-\x-sqrt(\x*\x-6*\x+7)});
  \draw[typetwo] (1,{3-sqrt(2)})--(2,2);

  % Type 2: f(u_{i3}^-(x))=u_{i3}^-(x), 2<x<=9/2.
  \draw[typetwo] (2,2)--(4,{4-2*sqrt(2)});
  \draw[typetwo] plot[domain=4:4.5,samples=60]
    (\x,{\x-sqrt(\x*\x-8)});

  % Type 2: f(u_{i1}^-(x))=u_{i1}^-(x)+2, x>=9/2.
  \draw[typetwo] plot[domain=4.5:6,samples=100]
    (\x,{\x+2-sqrt(\x*\x+4*\x-8)});

  % Coincident type-2 values:
  % u_{i1}^-=-1 on (-1,1], and u_{i3}^+=1 on [-1,1).
  % u_{i1}^+=-2 for x<-2; u_{i4}^-=4 for x>4.
  \draw[typetwo] (-1,1)--(1,1);
  \draw[typetwo] (-3,0)--(-2,0) (4,0)--(6,0);

  % Type 1: f(x).
  \draw[typeone] (-3,0)--(-2,0)--(-1,1)--(1,1)
    --(2,2)--(4,0)--(6,0);

  % Type 3: height 1 for x in [-1,1] or x=9/2;
  % height 0 for x<=-2 or x>=4.
  \draw[typethree] (-3,0)--(-2,0) (-1,1)--(1,1) (4,0)--(6,0);
  \foreach \p in {(-1,1),(1,1),(4.5,1)}
    \fill[red!85!black] \p circle (1.65pt);
  \draw[gray!55,densely dotted] (4.5,0)--(4.5,.97);
  \node[below] at (4.5,-.035) {$\frac92$};

  % Legend, matching the reference image.
  \draw[typeone] (.6,-.52)--(1.1,-.52);
  \node[anchor=west] at (1.3,-.52) {Type 1: $f(x)$};
  \draw[typetwo] (.6,-.80)--(1.1,-.80);
  \node[anchor=west] at (1.3,-.80)
    {Type 2: admissible $m_j u_{ij}^{\pm}(x)+b_j$};
  \draw[typethree] (.6,-1.08)--(1.1,-1.08);
  \node[anchor=west] at (1.3,-1.08)
    {Type 3: admissible $b_j$};
\end{tikzpicture}
\caption{$\M f(x)$ is the maximum of these candidates}
\label{fig:placeholder}
\end{figure}
In the last row, one may choose any $u\in I_j\setminus\{x\}$ because every included value is an average or its zero-length limit. Conversely, the existence of $a(x)$ shows that the list contains the maximum. 
Consequently, if $x\in I_i$ and $\mathcal C_i(x)$ denotes the included values,
\begin{equation}\label{eq:finite-list}
 \M f(x)=\max\mathcal C_i(x).
\end{equation}
Based on Table \ref{tab:candidates_of_Mf}, first cut at the endpoints of $I_i$ and at the zeros of the nonzero polynomials in the table as follows:

\begin{table}[ht]
\centering
\begin{tabular}{c|c}
Polynomials & Possible change its zero separates \\ \hline
$D_{ij}(x)$ & A real, nondegenerate endpoint $u_{ij}^{\pm}(x)$ appears or disappears\\
$D_{ij}(x)-(v-x)^2$ & The moving endpoint reaches a fixed boundary $v$ of $I_j$\\
$P_i(x)-b_jx-c_j$ when $m_j=0$ & The constant candidate $b_j$ becomes available at $x$
\end{tabular}
\vspace{2mm}
\caption{Polynomials needed for smooth partition}
\label{tab:polynomial_cuts}
\end{table}

In Table \ref{tab:candidates_of_Mf}, while the first type candidates always exist, the second type candidates need to satisfy two conditions: $D_{ij}(x)\geq0$ and $u_{ij}^\pm(x)\in\overline{I_j}$. These two conditions corresponding to two families of polynomials in Table \ref{tab:polynomial_cuts}, respectively: The first one is $D_{ij}(x)$, whose zeros determine the existence of $u_{ij}^\pm(x)$; the second one is $D_{ij}(x)-{(v-x)}^2$, where $v$ is a boundary point of $I_j$. Its zeros satisfy $v=x\pm\sqrt{D_{ij}(x)}=u_{ij}(x)$, so its location of zeros determine whether $u_{ij}^\pm(x)\in\overline{I_j}$ or not. As for the third type candidates, its condition is $P_i(x)-b_j(x)-c_j=0$ and $m_j=0$, so the location of zeros of the third family of polynomials in Table \ref{tab:polynomial_cuts} suffices to determine its existence. Hence, with these finitely many polynomials mentioned in Table \ref{tab:polynomial_cuts}, we can use their zeros and endpoints of intervals $I_j$ to partition the real numbers into finitely many intervals such that, on each interval $J$, $f|_J$ is affine and every candidate of $\M f$ appears either on the whole interval $J$ or nowhere on $J$. 

Next, we aim to show that the intersection of these candidate curves is either finite or a subinterval. Note that every candidate is both algebraic and analytic. Let $k_1(x),k_2(x)$ be two candidates and construct a nonzero irreducible polynomial $Q(x,y)\in\mathbb{R}[x,y]$ such that $Q(x,k_1(x)-k_2(x))=0$. Prove by contradiction. Suppose that there are infinitely many intersection points. Let $x$ be an intersection. Then, we have $k_1(x)=k_2(x)$ and $Q(x,0)=0$ has infinitely many solutions. Together with $Q\in\mathbb{R}[x,y]$, we must have $y\mid Q(x,y)$. Thus, by irreducibility, $Q(x,y)=cy$ for some constant $c\in\mathbb{R}$. In this case, $k_1(x)\equiv k_2(x)$. In summary, the intersection of candidates is either finite or a subinterval.

Since the intersection is either finite or a subinterval, there exists a finite refinement of the previous partition such that, on each interval $J$, $f|_J$ is affine and $\M f|_J$ is identical to a candidate of $\M f$ in Table \ref{tab:candidates_of_Mf} and thus $\M f$ is smooth on $J$. 

Since $\M f|_J$ is an algebraic function, via the algebraic relation, there exists a refinement of this partition such that on each interval $J$, $(\M f)'$ is monotone and 
\[(\M f)'|_J>0,\quad (\M f)'|_J<0,\quad\text{or}\quad (\M f)'|_J\equiv0.\]
\end{proof}
\begin{lemma}\label{lem:refined_partition}
Let $J$ be an open interval in the partition of Lemma \ref{lem:branches} with $(\M f)'|_J$ nonzero,
\[\mathcal{H}:=\{f(v)\mid v\text{ is a finite endpoint of an affine interval}\}\cup\{b_j\mid m_j=0\},\]
and $\mathcal{P}:=\{x\in\mathbb{R}\mid\M f(x)\in\mathcal{H},(\M f)'(x)\neq0\}$. Then, $\mathcal{P}$ is finite and on $J\setminus\mathcal{P}$, we have
\begin{equation}\label{eq:endpoint-signs}
 \begin{cases}
 f'(a(x))>0,\, a'(x)<0&\text{if }(\M f)'(x)<0,\\
 f'(a(x))<0,\, a'(x)<0&\text{if }(\M f)'(x)>0.
 \end{cases}
\end{equation}
At every finite partition point $x_0$, the limits $(\M f)'(x_0\pm)$
exist, are finite, and equal the corresponding one-sided derivatives.
If $x_0\in\Omega$, then
\begin{equation}\label{eq:upward-jump}
 (\M f)'(x_0-)\le(\M f)'(x_0+).
\end{equation}
\end{lemma}
\begin{proof}
For each $h\in\mathcal H$, remove the solutions of $\M f(x)=h$ from $J$.
There is at most one solution per $h\in\mathcal{H}$ because $\M f$ is strictly monotone on $J$. Since there are finitely many intervals, $\mathcal{P}$ is finite. After removing these solutions, on each remaining subinterval of $J$, continuity of $a(x)$ and $f(a(x))=\M f(x)\notin\mathcal H$ ensure that $a(x)$ stays in one $I_j$ with $m_j\ne0$.
The second derivative test at the maximizing endpoint now gives
\begin{equation}\label{eq:sgn_f'a}
0\ge\partial_{uu} A(a(x),x)=\frac{\frac{A(a(x),x)-f(a(x))}{x-a(x)}-f'(a(x))}{x-a(x)}+\frac{A(a(x),x)-f(a(x))}{{(x-a(x))}^2}
   =-\frac{f'(a(x))}{x-a(x)}
\end{equation}

because $A(a(x),x)=\M f(x)=f(a(x))$. On an interval $J$ where $(\M f)'\ne0$, Lemma~\ref{lem:endpoint} gives
\[
 a(x)=x-\frac{f(x)-\M f(x)}{(\M f)'(x)},\qquad f(a(x))=\M f(x).
\]
In particular, $a(x)<x$ when $(\M f)'<0$, and $a(x)>x$ when $(\M f)'>0$. Together with \eqref{eq:sgn_f'a}, this proves $f'(a(x))>0$ when $(\M f)'(x)<0$ and $f'(a(x))<0$ when $(\M f)'(x)>0$. Moreover, differentiating the identity $f(a(x))=\M f(x)$ directly yields
\[
 a'(x)=\frac{(\M f)'(x)}{f'(a(x))}<0
\]
 in both cases. Figure~\ref{fig:opposite-endpoint} illustrates the endpoint identity.
 
\begin{figure}[htbp]
\centering
\begin{tikzpicture}[x=3cm,y=3cm,>=stealth,font=\small]
 \pgfmathsetmacro{\a}{(1-sqrt(2))/2}
 \pgfmathsetmacro{\h}{(3-sqrt(2))/2}

 % Area under f on [a,1/2].
 \fill[blue!10]
   (\a,0)--(\a,\h)--(0,1)--(.5,.5)--(.5,0)--cycle;

 % Coordinate axes and graph of f.
 \draw[->] (-1.25,0)--(1.38,0) node[right] {$t$};
 \draw[->] (0,-.05)--(0,1.22);
 \draw[thick] (-1.2,0)--(-1,0)--(0,1)--(1,0)--(1.2,0);

 % Maximizing average and endpoints.
 \draw[blue!70!black,dashed] (-.65,\h)--(1.05,\h);
 \node[anchor=west] at (1.05,\h)
   {$\M f(\tfrac12)=\dfrac{3-\sqrt2}{2}$};
 \draw[dotted] (\a,0)--(\a,\h);
 \draw[dotted] (.5,0)--(.5,\h);
 \fill (\a,\h) circle(1.2pt);
 \fill (.5,.5) circle(1.2pt);

 % Labels.
 \node[below] at (-1,0) {$-1$};
 \node[below] at (\a,0) {$a(\tfrac{1}{2})$};
 \node[below] at (0,0) {$0$};
 \node[below] at (.5,0) {$x=\tfrac12$};
 \node[below] at (1,0) {$1$};
 \node[above right] at (.12,1.02) {$f(t)=(1-|t|)_+$};
\end{tikzpicture}
\caption{An example: $a(\tfrac{1}{2})=(1-\sqrt2)/2$ and $f(a)=\M f(\tfrac12)=(3-\sqrt2)/2$.}
\label{fig:opposite-endpoint}
\end{figure}
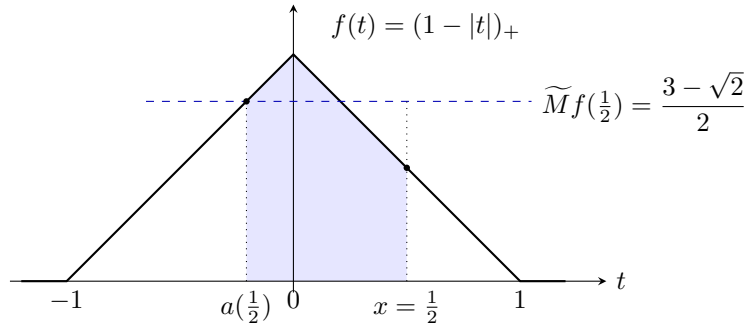

In Figure \ref{fig:opposite-endpoint}, the shaded area equals the area of the rectangle of width $\tfrac12-a$ and height $f(a)$. The maximizing endpoint lies on the increasing part of $f$ while $\tfrac{1}{2}$ lies on the decreasing part of $\M f$. By the maximality, $a(\tfrac{1}{2})$ will occur at the place where 
\[f(a)=\frac{1}{\tfrac{1}{2}-a}\int_a^{\tfrac{1}{2}}f.\]

The derivative $(\M f)'$ is monotone on every interval and satisfies
$|(\M f)'|\le\Lip(f)$ by \eqref{eq:basic-Lip}; hence its one-sided
limits at finite endpoints exist and are finite. For $h>0$ small,
the mean value theorem gives
\[
 \frac{\M f(x_0+h)-\M f(x_0)}h=(\M f)'(\xi_h),
 \qquad x_0<\xi_h<x_0+h.
\]
Letting $h\downarrow0$ identifies the right derivative with the right
limit; the left side is identical.

For $x_0\in\Omega$, choose a maximizing interval $[u,v]$ of $x_0$ and define,
for $t$ near $x_0$,
\[
 Q(t):=\begin{cases}
 \displaystyle\frac1{v-u}\int_u^v f,&u<x_0<v,\\[2mm]
 \displaystyle\frac1{t-u}\int_u^t f,&x_0=v,\\[2mm]
 \displaystyle\frac1{v-t}\int_t^v f,&x_0=u.
 \end{cases}
\]
Then $Q\in C^1$ at $x_0$, $Q\le\M f$, and $Q(x_0)=\M f(x_0)$.
For $h>0$, we have
\[\frac{\M f(x_0-h)-\M f(x_0)}{-h}\leq\frac{Q(x_0-h)-Q(x_0)}{-h},\quad\frac{Q(x_0+h)-Q(x_0)}h
 \le\frac{\M f(x_0+h)-\M f(x_0)}h.\]
Thus $(\M f)'(x_0-)\le Q'(x_0)\le(\M f)'(x_0+)$, proving
\eqref{eq:upward-jump}; see Figure~\ref{fig:touching-slopes} for a visual illustration.

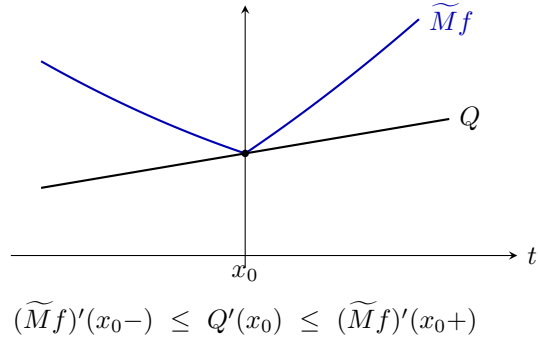
\begin{figure}[htbp]
\centering
\begin{tikzpicture}[x=2cm,y=1.35cm,>=stealth,font=\small]
 \draw[->] (-1.55,0)--(1.8,0) node[right] {$t$};
 \draw[->] (0,-.12)--(0,2.45);
 \draw[thick,blue!70!black,domain=-1.35:0,samples=30]
   plot (\x,{1-0.5*\x+0.125*\x*\x});
 \draw[thick,blue!70!black,domain=0:1.15,samples=30]
   plot (\x,{1+\x+0.125*\x*\x});
 \draw[thick,domain=-1.35:1.35] plot (\x,{1+0.25*\x});
 \fill (0,1) circle(1.4pt);
 \node[below] at (0,0) {$x_0$};
 \node[right] at (1.35,1.3375) {$Q$};
 \node[right,blue!70!black] at (1.15,2.315) {$\M f$};
 \node[below,align=center] at (0,-.35)
 {$ (\M f)'(x_0-)\ \le\ Q'(x_0)\ \le\ (\M f)'(x_0+)$};
\end{tikzpicture}
\caption{Schematic of the touching inequality $Q\le\M f$,
$Q(x_0)=\M f(x_0)$. The left and right slopes of $\M f$ bound
$Q'(x_0)$. This picture illustrates the comparison argument, rather
than a particular maximal function.}
\label{fig:touching-slopes}
\end{figure}

\end{proof}

Let $\mathcal{P}'$ be the collection of endpoints of the finite partition provided by Lemma \ref{lem:branches}, define the refined partition $\mathcal{P}^*:=\mathcal{P}\cup\mathcal{P}'$, and 
\[E:=\{x\in\mathbb{R}\mid x\notin\mathcal{P},x\notin\mathcal{P}',(\M f)'(x)\neq0\}.\]
Write $E_+=\{x\in E\mid(\M f)'(x)>0\}$ and $E_-=\{x\in E\mid(\M f)'(x)<0\}$. The notation $a(x)$ will always mean the opposite endpoint of the unique maximizing interval at the evaluation point $x$.

\begin{lemma}\label{lem:injectivity}
The function $x\mapsto a(x)$ is injective on $E_-$ and on $E_+$. Its images from these two sets are disjoint.
\end{lemma}

\begin{proof}
Prove by contradiction. Suppose that $x_1<x_2$, $x_1,x_2\in E_-$, and $a(x_1)=a(x_2)$. Then,
\[\M f(x_1)=f(a(x_1))=f(a(x_2))=\M f(x_2).\]
By \eqref{eq:zero-contact}, we know that $(\M f)'(x_i)<0$ implies $f(x_i)<\M f(x_i)$ and thus $a(x_i)<x_i$ for $i=1,2$. Figure \ref{figure:position} illustrates an idea of the relation between $x_1,x_2$, and $a(x_1)$.

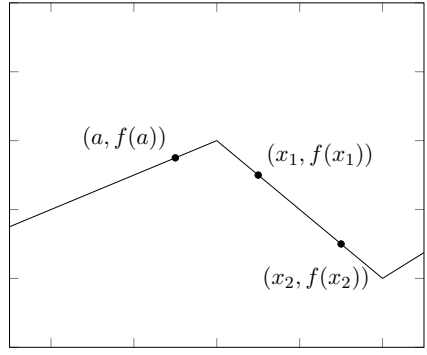
\begin{figure}[ht]
\centering
\begin{tikzpicture}[scale=.8]
        \begin{axis}[
            xmin=-2.5,xmax=2.5,ymin=0,ymax=5,
            xticklabels={,,,,,},
            yticklabels={,,,,,,},
        ]
            \addplot+[-,black, domain=-2.5:0, mark=none, samples=200] {0.5*x+3};
            \addplot+[-,black, domain=0:2, mark=none, samples=200] {-1*x+3};
            \addplot+[-,black, domain=2:2.5, mark=none, samples=200] {0.75*x-0.5};
            \draw[fill] (axis cs:{0.5,2.5}) circle [radius=1.5pt] node[above right] {$(x_1,f(x_1))$};
            \draw[fill] (axis cs:{1.5,1.5}) circle [radius=1.5pt]; 
            \node at (1.2,1) {$(x_2,f(x_2))$};
            \draw[fill] (axis cs:{-0.5,2.75}) circle [radius=1.5pt] node[above left] {$(a,f(a))$};
        \end{axis}
\end{tikzpicture}
\caption{The position of $x_1,x_2$, and $a:=a(x_1)=a(x_2)$}
\label{figure:position}
\end{figure}

Since $f$ is continuous and $f(x_2)<\M f(x_2)$, removing a short piece at the right end of $[a(x_2),x_2]$ removes only values below its average. Let the shortened interval be $[a(x_1),b]$ with $x_1<b$. Then, we have $x_1\in [a(x_1),b]$ but
\[\frac{1}{b-a(x_1)}\int_{a(x_1)}^bf>\frac{1}{x_2-a(x_2)}\int_{a(x_2)}^{x_2}f=\M f(x_2)=\M f(x_1),\]
which contradicts the definition of the maximal operator. By symmetry, $x\mapsto a(x)$ is also injective on $E_+$.

Additionally, by~\eqref{eq:endpoint-signs},
\[
 \{a(x):x\in E_-\}\subset\{t:f'(t)>0\},
 \qquad
 \{a(x):x\in E_+\}\subset\{t:f'(t)<0\}.
\]
Thus, the images are disjoint.
\end{proof}

\subsection{The derivative estimate}\label{sec:proof:piecewise_affine}
In this section, we will use the refined partition $\mathcal{P}^*$ to obtain the desired derivative estimates. We first justify the decay estimates that make all the following integrals finite. Choose $R>0$ with $\supp f\subset[-R,R]$, and fix a differentiability point $x>R$ of $\M f$. Since $f(x)=0<\M f(x)$, a maximizing interval of $x$ exists. Moreover, the maximizing interval has the form $[a(x),x]$, with $-R\le a(x)<R$. For a sufficiently small $h>0$, we have $\int_{x-h}^xf(y)dy=0$ and thus
\begin{equation*}
\begin{split}
\M f(x-h)-\M f(x)&\geq\frac{1}{(x-h)-a(x)}\int_{a(x)}^{x-h}f(y)dy-\frac{1}{x-a(x)}\int_{a(x)}^xf(y)dy\\
&=\frac{h}{(x-h-a(x))(x-a(x))}\int_{a(x)}^xf(y)dy.
\end{split}
\end{equation*}
i.e. $(\M f)'(x)<0$.
Thus, for $x>R$, by the definition of the maximal operator and Lemma \ref{lem:endpoint}, 
\[\M f(x)\leq\frac{\norm{f}_1}{x-a(x)}\leq\frac{\norm{f}_1}{x-R},\quad|(\M f)'(x)|\leq\frac{|\M f(x)|}{|x-a(x)|}\leq\frac{\norm{f}_1}{{(x-R)}^2}.\]
Reflection gives the similar estimates for the case $x<-R$. In summary,
\begin{equation}\label{eq:decay_estimate}
\M f(x)\le\frac{\norm{f}_1}{|x|-R},\quad|(\M f)'(x)|\le\frac{\norm{f}_1}{(|x|-R)^2}\quad\text{for all}\quad |x|>R.
\end{equation}

On $[-R-1,R+1]$, $|(\M f)'(x)|\le\Lip(f)$ by~\eqref{eq:basic-Lip}; outside $[-R-1,R+1]$ the preceding decay estimate \eqref{eq:decay_estimate} applies. Thus $\int_{\R}|(\M f)'(x)|^r\,dx<\infty$ for every $r>1/2$. Finally, the nonzero slopes of $f$ form a finite set, and Lemma~\ref{lem:refined_partition} ensures that $f'(a(x))$ is one of these slopes on $E$. Therefore $|f'(a(x))|$ is bounded above and away from zero for all $x\in E$. These facts establish the finiteness of all the integrals used below, without invoking the sharp inequality we are about to prove.

Our goal is the following proposition:
\begin{proposition}\label{prop:main_estimates}
Let $1<p<\infty$ and $f$ be nonnegative, continuous, compactly supported, and affine on each of finitely many intervals. Then, 
\begin{align}
 &\int_{\R} f'(x)(\M f)'(x)|(\M f)'(x)|^{p-2}\,dx
 \notag\\
 &\quad\ge
 \left(2-\frac1p\right)\int_{\R}|(\M f)'(x)|^p\,dx
 +\frac{p-1}{p}\int_E
 \frac{|(\M f)'(x)|^{p+1}}{|f'(a(x))|}\,dx,
 \label{eq:coercivity}\\[1mm]
 &\int_E |f'(a(x))|^{p-1}|(\M f)'(x)|\,dx
 \le\int_{\R}|f'(x)|^p\,dx.
 \label{eq:budget}
\end{align}
\end{proposition}

To prove~\eqref{eq:coercivity}, we use a signed identity that treats increasing and decreasing intervals together. Choose the refined partition $\mathcal{P}^*$, and write its endpoints as
\[
 -\infty=t_0<t_1<\cdots<t_m<t_{m+1}=+\infty.
\]
On each open interval $(t_j,t_{j+1})$, the derivative $(\M f)'(x)$ has one fixed strict sign or is identically zero. The sets $E_-$ and $E_+$ are the unions of the intervals of the first two types. 

Work first on one nonconstant interval. The second-derivative formula~\eqref{eq:weigt-identities}, reads
\[
 f'(x)=(2-a'(x))(\M f)'(x)
 +(x-a(x))(\M f)''(x).
\]
It holds for either sign of $(\M f)'(x)$: $x-a(x)>0$ on a decreasing interval and $x-a(x)<0$ on an increasing interval. Multiply by $(\M f)'(x)|(\M f)'(x)|^{p-2}$ and integrate on a compact subinterval $[u,v]$. We obtain 
\begin{equation}\label{eq:int_weigt_identities}
\begin{split}
&\int_u^vf'(x)(\M f)'(x)|(\M f)'(x)|^{p-2}dx\\
&=\int_u^v(2-a'(x))|(\M f)'(x)|^pdx+\int_u^v(x-a(x))(\M f)''(x)(\M f)'(x)|(\M f)'(x)|^{p-2}dx.
\end{split}
\end{equation}
Observe that
\[
 \frac{d}{dx}|(\M f)'(x)|^p
 =p(\M f)'(x)|(\M f)'(x)|^{p-2}(\M f)''(x),
\]
which could be related to the second term of the right-hand side of equation \eqref{eq:int_weigt_identities}.
Thus, by integration by parts,
\begin{equation}\label{eq:branch-energy}
\begin{split}
 &\int_u^v f'(x)(\M f)'(x)|(\M f)'(x)|^{p-2}\,dx\\
 &=\int_u^v(2-a'(x))|(\M f)'(x)|^pdx+\int_u^v\frac{x-a(x)}{p}\frac{d}{dx}|(\M f)'(x)|^pdx\\
 &\quad=\left(2-\frac1p\right)
       \int_u^v |(\M f)'(x)|^p\,dx
       +\frac{1-p}{p}
       \int_u^va'(x)|(\M f)'(x)|^p\,dx\\
 &\qquad\quad-
 \left[\frac{(\M f(x)-f(x))(\M f)'(x)|(\M f)'(x)|^{p-2}}{p}\right]_{x=u}^{x=v}.
\end{split}
\end{equation}
Here we used
\[
 (x-a(x))(\M f)'(x)=f(x)-\M f(x)
\]
by \eqref{eq:local-identities}. Also, by~equations \eqref{eq:weigt-identities} and~\eqref{eq:endpoint-signs}, for every $x\in E$, 
\begin{equation}\label{eq:charge-pointwise}
 -a'(x)|(\M f)'(x)|^p
 =\frac{|(\M f)'(x)|^{p+1}}{|f'(a(x))|}
 \qquad(x\in E).
\end{equation}
\smallskip\noindent\emph{Passing to the endpoints of a smooth interval.}
We first let $u\downarrow t_j$ and $v\uparrow t_{j+1}$ within one nonconstant interval. All integrands in~\eqref{eq:branch-energy} are integrable, by the bounds at the beginning of this subsection and~\eqref{eq:charge-pointwise}. Thus the bulk integrals converge by dominated convergence theorem. Next, we will examine the contribution of a finite/infinite partition point $t_j$ in equation \eqref{eq:branch-energy} below.

Consider the integral in \eqref{eq:branch-energy} and let $u\downarrow t_j$ and $v\uparrow t_{j+1}$, where $t_j$ are the refined partition points in $\mathcal{P}^*$. For each finite partition point $t_j$, define
\[R_j^\pm:=\pm\frac{\M f(t_j)-f(t_j)}{p}(\M f)'(t_j\pm)\bigl|(\M f)'(t_j\pm)\bigr|^{p-2},\qquad 1\le j\le m,\]
where $r|r|^{p-2}$ is interpreted as $0$ at $r=0$.
At the infinite endpoints, define
\begin{equation*}
\begin{split}
R_0^+&:=\lim_{x\to-\infty}\frac{\M f(x)-f(x)}{p}(\M f)'(x)|(\M f)'(x)|^{p-2},\\
R_{m+1}^-&:=-\lim_{x\to+\infty}\frac{\M f(x)-f(x)}{p}(\M f)'(x)|(\M f)'(x)|^{p-2}.
\end{split}
\end{equation*}

\begin{lemma}\label{lem:boundary-contributions}
The boundary term in \eqref{eq:branch-energy} equals $R_j^+ + R_{j+1}^-$. Moreover, for $1\leq j\leq m$,
\[
R_j^+ + R_j^-\ge0,
\qquad
R_0^+=R_{m+1}^-=0.
\]
Consequently, the total boundary contribution is nonnegative:
\[
\sum_{j=0}^{m}\bigl(R_j^+ + R_{j+1}^-\bigr)
=
\sum_{j=1}^{m}\bigl(R_j^+ + R_j^-\bigr)
\ge0.
\]
\end{lemma}

\begin{proof}
The derivative $(\M f)'$ is bounded and monotone on each
partition interval, so its one-sided traces at finite
partition points exist. Since $\M f-f$ is continuous,
the boundary term in \eqref{eq:branch-energy} satisfies
\[
-\left[
\frac{\M f(x)-f(x)}{p}
(\M f)'(x)|(\M f)'(x)|^{p-2}
\right]_{x=u}^{x=v}
\longrightarrow R_j^+ + R_{j+1}^-.
\]

Fix a finite partition point $t_j$. If $\M f(t_j)=f(t_j)$, then $R_j^+=R_j^-=0$. Otherwise,
$t_j\in\Omega$, and \eqref{eq:upward-jump} gives
\[
(\M f)'(t_j-)\le(\M f)'(t_j+).
\]
Since $r\mapsto r|r|^{p-2}$ is increasing for $p>1$,
\[
\begin{aligned}
R_j^+ + R_j^-
&=
\frac{\M f(t_j)-f(t_j)}{p}
\Bigl(
(\M f)'(t_j+)|(\M f)'(t_j+)|^{p-2}
\\
&\hspace{42mm}
-
(\M f)'(t_j-)|(\M f)'(t_j-)|^{p-2}
\Bigr)
\ge0.
\end{aligned}
\]

Finally, choose $R>0$ such that
$\operatorname{supp}f\subset[-R,R]$.
By the tail estimates \eqref{eq:decay_estimate}, for $|x|>R$,
\[
\left|
\frac{\M f(x)-f(x)}{p}
(\M f)'(x)|(\M f)'(x)|^{p-2}
\right|
\le
\frac{\|f\|_1^p}{p(|x|-R)^{2p-1}}
\longrightarrow0.
\]
Thus $R_0^+=R_{m+1}^-=0$. Summing the interval
contributions yields the stated identity and inequality.
\end{proof}

Now, we are ready to prove Proposition \ref{prop:main_estimates}.
\begin{proof}[Proof of Proposition \ref{prop:main_estimates}]
Now sum~\eqref{eq:branch-energy} over the finitely many nonconstant intervals $(t_j,t_{j+1})$, after taking the endpoint limits on each interval separately. Every finite endpoint appears exactly through
\begin{align}\label{eq:switch}
 \frac{\M f(t_j)-f(t_j)}{p}
 \Big(& (\M f)'(t_j+) |(\M f)'(t_j+)|^{p-2}-(\M f)'(t_j-) |(\M f)'(t_j-)|^{p-2}\Big),
\end{align}
and all those contributions are nonnegative. There is no contribution from infinity. Outside $E$, the derivative $(\M f)'(x)$ vanishes almost everywhere, so the omitted intervals and finitely many partition points contribute zero to the integrals. Using \eqref{eq:branch-energy} and \eqref{eq:charge-pointwise}, we conclude that
\begin{align*}
 &\int_{\R} f'(x)(\M f)'(x)|(\M f)'(x)|^{p-2}\,dx\\
 &\quad\ge\left(2-\frac1p\right)\int_{\R}|(\M f)'(x)|^p\,dx
 +\frac{p-1}{p}\int_E
 \frac{|(\M f)'(x)|^{p+1}}{|f'(a(x))|}\,dx.
\end{align*}
This completes the proof of \eqref{eq:coercivity}. 

For~\eqref{eq:budget}, use the first formula in~\eqref{eq:weigt-identities} and change variables on each smooth interval $J\subset E$:
\begin{align*}
 \int_J |f'(a(x))|^{p-1}|(\M f)'(x)|\,dx
 &=\int_J |f'(a(x))|^p|a'(x)|\,dx\\
 &=\int_{\{a(x):x\in J\}} |f'(t)|^p\,dt.
\end{align*}
Lemma~\ref{lem:injectivity} makes these image sets pairwise disjoint, also between the two signs of $(\M f)'$. Their sum is bounded by $\int_{\R}|f'(t)|^p\,dt$, proving~\eqref{eq:budget}.
\end{proof}

Next, we will first prove Theorem \ref{thm:main} for those functions $f\geq0$, which are continuous, compactly supported, and affine on each of finitely many intervals. Then, apply an approximation argument to show that Theorem \ref{thm:main} is valid for general function $f$ satisfying the assumptions mentioned in Theorem \ref{thm:main}.
\begin{proof}[Proof of \eqref{eq:main} in Theorem \ref{thm:main} for piecewise-affine $f$]
H\"older's inequality and~\eqref{eq:budget} give
\begin{align*}
 \norm{(\M f)'}_p^p
 &=\int_E
 \left(\frac{|(\M f)'(x)|^{p+1}}{|f'(a(x))|}\right)^{(p-1)/p}
 \left(|f'(a(x))|^{p-1}|(\M f)'(x)|\right)^{1/p}\,dx\\
 &\le
 \left(\int_E\frac{|(\M f)'(x)|^{p+1}}{|f'(a(x))|}\,dx\right)^{(p-1)/p}
 \norm{f'}_p.
\end{align*}
If $\norm{(\M f)'}_p=0$, the conclusion is immediate. Otherwise $\norm{f'}_p>0$, and the preceding inequality gives
\begin{equation*}
 \int_E\frac{|(\M f)'(x)|^{p+1}}{|f'(a(x))|}\,dx
 \ge
 \frac{\norm{(\M f)'}_p^{p^2/(p-1)}}
      {\norm{f'}_p^{p/(p-1)}}.
\end{equation*}
A second application of H\"older to the left side of~\eqref{eq:coercivity} yields
\begin{align*}
 \norm{f'}_p\norm{(\M f)'}_p^{p-1}
 \ge{}&\left(2-\frac1p\right)\norm{(\M f)'}_p^p\\
 &+\frac{p-1}{p}
 \frac{\norm{(\M f)'}_p^{p^2/(p-1)}}
      {\norm{f'}_p^{p/(p-1)}}.
\end{align*}
After division and multiplication by $p$, this becomes
\begin{equation*}
 (2p-1)\frac{\norm{(\M f)'}_p}{\norm{f'}_p}
 +(p-1)\left(\frac{\norm{(\M f)'}_p}{\norm{f'}_p}\right)^{(2p-1)/(p-1)}
 \le p.
\end{equation*}
This gives \eqref{eq:gamma}. Additionally, since $(2p-1)x+(p-1)x^{(2p-1)/(p-1)}$ is monotone increasing and by intermediate value theorem, \eqref{eq:gamma} has a unique solution $\gamma_p$ on $[0,1]$. Thus, this proves \eqref{eq:main} for piecewise-affine inputs. 
\end{proof}

\subsection{Approximation}\label{sec:approximation}
We will use lower semicontinuity of the Riesz variation under pointwise convergence: Assume $u_j(x)\to u(x)\in\R$ for every $x$ and $1\le p<\infty$. Fix a partition $P=\{x_0<\dotsb<x_m\}$ and write 
\[S_P(u):=\left(\sum_{j=1}^m\frac{{|u(x_j)-u(x_{j-1})|}^p}{{|x_j-x_{j-1}|}^{p-1}}\right)^\frac{1}{p}.\]
Then, by the convergence and the definition, for all $P$,
\begin{equation}\label{eq:variation-lsc}
S_P(u)=\liminf_{n\to\infty}S_P(u_n)\leq\liminf_{n\to\infty}\Var_p(u_n)\Rightarrow\Var_p(u)\leq\liminf_{n\to\infty}\Var_p(u_n).
\end{equation}
Indeed, pass to the limit along each fixed finite partition in~\eqref{eq:riesz}, and then take the supremum over partitions. For $p>1$, a finite right side gives local absolute continuity of the actual pointwise limit and the identity $\Var_p(u)=\norm{u'}_p$, by~\cite[Chapter IX \textsection 4, Theorem 7]{N}. 

\begin{proof}[Proof of \eqref{eq:main} in Theorem \ref{thm:main}]
Assume that $f\ge0$ is compactly supported and belongs to $W^{1,p}(\R)$, and use its locally absolutely continuous representative. Choose a bounded interval $[A,B]$ containing
$\operatorname{supp}f$ in its interior. Let $f_j$ be the
piecewise-linear interpolants on uniform partitions of $[A,B]$
whose mesh sizes tend to zero, extended by zero outside $[A,B]$.
Uniform continuity gives $f_j\to f$ uniformly. On each partition interval $I$, absolute continuity gives
\[
 f_j'|_I=\frac{1}{|I|}\int_I f'(t)\,dt.
\]
The corresponding cell-average operators are contractions on
$L^p([A,B])$ by Jensen's inequality and converge uniformly to the
identity on continuous functions. By density of continuous
functions in $L^p([A,B])$; see~\cite[Theorem 3.14]{Ru},
it follows that $f_j'\to f'$ in $L^p(\R)$. Since $f$ is bounded, the elementary inequality
\begin{equation}\label{eq:uniform-M}
 \norm{\M f-\M g}_\infty\le\norm{f-g}_\infty
\end{equation}
gives uniform convergence $\M f_j\to\M f$. Consequently, since $f_j$ satisfies the piecewise affine assumption, by the proof at the end of section \ref{sec:proof:piecewise_affine}, we know that $f_j$ satisfy \eqref{eq:main} and thus
\begin{equation}\label{eq:variation_bound}
\Var_p(\M f)
 \le\liminf_{j\to\infty}\Var_p(\M f_j)\leq\liminf_{j\to\infty}\gamma_p{\|f_j'\|}_p=\gamma_p\norm{f'}_p.
\end{equation}
In particular, since $p>1$ and $\Var_p(\M f)<\infty$, $\M f\in AC_{\mathrm{loc}}(\R)$. Thus, $(\M f)'\in L^p(\R)$ and $\|(\M f)'\|_p=\Var_p(\M f)$. Together with \eqref{eq:variation_bound}, this proves \eqref{eq:main} for compactly supported $f\in W^{1,p}$.

Next, suppose that $f\ge0$ is bounded, locally absolutely continuous, and $f'\in L^p$, let $\eta_R(x)$ equal one on $[-R,R]$, vanish outside $[-2R,2R]$, and be affine on the two transition intervals. Then
\[\norm{(\eta_Rf)'}_p\le\norm{f'}_p+\norm{f}_\infty(2R^{1-p})^{1/p}.\]
The error term $\norm{f}_\infty(2R^{1-p})^{1/p}$ tends to zero as $R\to\infty$. Since $0\le\eta_R(x)f(x)\le f(x)$ and these functions eventually agree with $f$ on each fixed averaging interval,
\begin{equation}\label{eq:cutoff-M}
 \M(\eta_R f)(x)\longrightarrow\M f(x).
\end{equation}
All these maximal functions are finite, as they are bounded by $\norm{f}_\infty$. By the same approximation argument in \eqref{eq:variation_bound}, since $\eta_Rf$ has compact support, we obtain
\[\Var_p(\M f)\leq\liminf_{R\to\infty}\Var_p(\M\eta_R f)\leq\liminf_{R\to\infty}\gamma_p\norm{(\eta_Rf)'}_p=\gamma_p\norm{f'}_p.\]
Similarly, by the boundedness of $\Var_p(\M f)$, we know $\norm{(\M f)'}_p=\Var_p(\M f)$, which completes the proof of \eqref{eq:main} for bounded $f$.

For a general nonnegative function $f$ in Theorem~\ref{thm:main}, suppose that $\M f(x_0)<\infty$. For every finite interval $[a,b]$ containing $x_0$ and $h\in\mathbb{R}$, by H\"{o}lder's inequality,
\begin{equation*}
\begin{split}
&\frac1{b-a}\int_{a+h}^{b+h}f-\frac1{b-a}\int_a^bf=\frac{1}{b-a}\int_a^b\int_t^{t+h}f'(u)dudt\\
&\leq\frac{1}{b-a}\int_a^b{\|f'\|}_p{|h|}^{1-1/p}dt=\|f'\|_p|h|^{1-1/p}.
\end{split}
\end{equation*}
Taking supremum over all intervals $[a,b]$ containing $x_0$, we know that 
\[\M f(x_0+h)=\sup_{x_0\in [a,b]}\frac{1}{b-a}\int_{a+h}^{b+h}f\leq\M f(x_0)+\|f'\|_p|h|^{1-1/p}.\]
Additionally, replacing $h$ by $-h$ and $x_0$ by $(x_0+h)$, we also have 
\[\M f(x_0)=\M f((x_0+h)-h)\leq\M f(x_0+h)+\|f'\|_p|h|^{1-1/p}.\]
To sum up,
\[|\M f(x_0+h)-\M f(x_0)|\le\|f'\|_p|h|^{1-1/p}.\]
Thus, once $\M f$ is finite at one point, it is finite and continuous everywhere. Apply the bounded-input estimate to $f_K(x)=\min\{f(x),K\}$. We have $|f_K'(x)|\le|f'(x)|$ almost everywhere and $\M f_K(x)\uparrow\M f(x)$.
Applying the same approximation argument in \eqref{eq:variation_bound}, since $f_K$ are bounded, 
\[\Var_p(\M f)\leq\liminf_{K\to\infty}\Var_p(\M f_K)\leq\liminf_{K\to\infty}\gamma_p\norm{f_K'}_p=\gamma_p\norm{f'}_p.\]
Similarly, by the boundedness of $\Var_p(\M f)$, we know $\norm{(\M f)'}_p=\Var_p(\M f)$, which completes the proof of \eqref{eq:main} for general nonnegative $f$.

Replacing a signed function $f$ by $|f|$ completes the upper bound, since $\M|f|=\M f$ and $|(|f|)'(x)|\le|f'(x)|$ almost everywhere.
\end{proof}
\subsection{Sharpness}\label{sec:sharpness}
For $0<\alpha\le1$, by the first derivative test, the expression minimized in~\eqref{eq:holder-constant} satisfies
\[
 (1+\alpha)t^\alpha+\alpha t^{1+\alpha}-1=0.
\]
The left-hand side is strictly increasing, negative at zero, and positive at one. Hence, there is a unique minimizer $t_\alpha\in(0,1)$, with
\begin{equation*}
 (1+\alpha)t_\alpha^\alpha+\alpha t_\alpha^{1+\alpha}=1,
 \qquad H_\alpha=\frac{1+t_\alpha^{1+\alpha}}{(1+\alpha)(1+t_\alpha)}=t_\alpha^\alpha.
\end{equation*}
In particular,
\begin{equation}\label{eq:H-equation}
 (1+\alpha)H_\alpha+\alpha H_\alpha^{1+1/\alpha}=1.
\end{equation}
Taking $\alpha=1-1/p$ gives $H_{1-1/p}=\gamma_p$, proving~\eqref{eq:critical-holder}.
\begin{proof}[Proof of the sharpness of Theorem \ref{thm:main}]
Fix $1-1/p<\theta\le1$ and put
\begin{equation*}
 f_\theta(x)=\max\{1-{|x|}^\theta,0\}.
\end{equation*}
Then $f_\theta\in W^{1,p}(\R)$ and
\begin{equation*}
 \norm{f_\theta'}_p^p
 =2\theta^p\int_0^1x^{p(\theta-1)}\,dx
 =\frac{2\theta^p}{p\theta-p+1}.
\end{equation*}
For a symmetrically decreasing function supported on $[-1,1]$, a maximizing interval for $x>0$ can be taken as $[-tx,x]$, with $0\le t\le1$. Indeed, on the right hand of $x$, $f_\theta$ is strictly decreasing and thus we may consider a maximizing interval in the form $[a,x]$. Additionally, since $f_\theta$ is symmetrically decreasing, we have $-x\leq a\leq0$. i.e. the maximizing interval can be written as $[-tx,x]$ with $0\leq t\leq1$.

For $0<x<1$, direct integration therefore gives
\begin{equation*}
\begin{split}
&\M f_\theta(x)
 =\max_{0\leq t\leq1}\frac{1}{(t+1)x}\int_{-tx}^xf_\theta(y)dy=1-x^\theta\min_{0\le t\le1}
 \frac{1+t^{1+\theta}}{(1+\theta)(1+t)}
 =1-H_\theta x^\theta\\
&\Rightarrow(\M f_\theta)'(x)=-\theta H_\theta x^{\theta-1}
\end{split}
\end{equation*}
Thus, by a straightforward estimate and \eqref{eq:main}, 
\begin{equation*}
\begin{split}
&\norm{(\M f_\theta)'}_p^p\geq2\int_0^1{(\theta H_\theta x^{\theta-1})}^pdx=\frac{2\theta^pH_\theta^p}{p(\theta-1)+1}=H_\theta^p\norm{f_\theta'}_p^p\\
&\Rightarrow H_\theta\leq\frac{\norm{(\M f_\theta)'}_p}{\norm{f_\theta'}_p}\leq\gamma_p
\end{split}
\end{equation*}

As $\theta\downarrow1-1/p$, continuity of the minimum yields $H_\theta\to\gamma_p$, so we know that
\begin{equation*}
 \lim_{\theta\downarrow1-1/p}
 \frac{\norm{(\M f_\theta)'}_p}{\norm{f_\theta'}_p}
 =\gamma_p.
\end{equation*}
Together with the proof in Section \ref{sec:approximation}, we complete the proof of Theorem~\ref{thm:main}.
\end{proof}

\section{The optimal constant and its endpoint limits}\label{sec:endpoints}

We first prove the monotonicity, limits, and asymptotic expansion of the optimal constant. We then show how the quantitative bounds of the classical endpoint theorems arise as limits of the finite-$p$ estimate.

\subsection{The optimal constant \texorpdfstring{$\gamma_p$}{gamma p}}
Now, we are ready to prove Proposition \ref{prop:coefficient}. 
\begin{proof}[Proof of Proposition \ref{prop:coefficient}]
The identity $\gamma_p=H_{1-1/p}$ was proved in~\eqref{eq:H-equation}. To see monotonicity, write
\[
 H_\alpha=\min_{0\leq t\leq 1}\frac{1+t^{1+\alpha}}{(1+t)(1+\alpha)}=\min_{0\le t\le1}
 \frac1{1+t}\int_{-t}^1|u|^\alpha\,du.
\]
Write $t_\alpha\in [0,1]$ be the minimizer of $\displaystyle g_\alpha(t):=\dfrac1{1+t}\int_{-t}^1|u|^\alpha\,du$ and thus $H_\alpha=g_\alpha(t_\alpha)$. 
For each $t$, this average $g_\alpha(t)$ is strictly decreasing in $\alpha$. Thus, for $\alpha<\alpha'$, we have
\[H_\alpha=g_\alpha(t_\alpha)>g_{\alpha'}(t_\alpha)\geq g_{\alpha'}(t_{\alpha'})=H_{\alpha'}\]
Therefore, $H_\alpha$ and thus $\gamma_p$ are strictly decreasing.

The expression in~\eqref{eq:holder-constant} extends continuously to $\alpha=0$, uniformly for $t\in[0,1]$, with value one. Hence $H_\alpha\to1$ as $\alpha\downarrow0$. At $\alpha=1$,~\eqref{eq:H-equation} gives $H_1^2+2H_1=1$, so $H_1=\sqrt2-1$. This proves~\eqref{eq:limits}.

For the first-order expansion, differentiate~\eqref{eq:H-equation} at $\alpha=1$. With $r=\sqrt2-1$, the implicit function theorem gives
\[
 \left.\frac{dH_\alpha}{d\alpha}\right|_{\alpha=1}
 =-\frac{r+r^2-r^2\log r}{2(1+r)}
 =-\frac{1-r-r^2\log r}{2(1+r)}.
\]
Taylor expansion at $\alpha=1$, using $\alpha=1-1/p$, proves~\eqref{eq:asymptotic} and completes Proposition~\ref{prop:coefficient}.
\end{proof}
\subsection{Endpoint behavior of \texorpdfstring{$\gamma_p$}{}}

By a limit argument, one may define $\gamma_1=1,\gamma_\infty=\sqrt{2}-1$. Then, both are optimal constants for the case $p=1$ and $p=\infty$, respectively \cite{ACP,AP}. Therefore, the quantitative endpoint result 
\begin{equation}\label{eq:gamma1inf}
\norm{(\M f)'}_1\leq\norm{f'}_1,\qquad \norm{(\M f)'}_\infty\leq(\sqrt{2}-1)\norm{f'}_\infty
\end{equation}
in \cite{ACP,AP} can be recovered by a limit argument as follows:
\begin{proof}[Proof of \eqref{eq:gamma1inf}]
First, consider the side $p\to\infty$. 

Let $f$ be compactly supported and Lipschitz. Then $f'\in L^p$ for every $p>1$. The basic estimate~\eqref{eq:basic-Lip} makes $(\M f)'$ bounded, and Theorem~\ref{thm:main} shows that $(\M f)'$ is in $L^p$. Therefore
\begin{align*}
 \norm{(\M f)'}_\infty
 &=\lim_{p\to\infty}\norm{(\M f)'}_p\\
 &\le\left(\lim_{p\to\infty}\gamma_p\right)
       \left(\lim_{p\to\infty}\norm{f'}_p\right)
 = (\sqrt2-1)\norm{f'}_\infty.
\end{align*}
For Lipschitz functions, these derivative norms equal the Lipschitz seminorms. This proves~\eqref{eq:endpoint-Lip} for compact support.

Next, apply this result to bounded Lipschitz functions by using the function $\eta_R$ from \eqref{eq:cutoff-M}. Finally, considering $f_K(x)=\min\{|f(x)|,K\}$ and letting $K\to\infty$ as we did in Section \ref{sec:approximation}, one can generalize the above result to all function $f\in\AC_\mathrm{loc}(\R),f'\in L^\infty(\R)$, and $\M f$ is finite at some point. This recovers the quantitative whole-line Lipschitz inequality~\eqref{eq:endpoint-Lip} in ~\cite{ACP} as a limit of Theorem~\ref{thm:main}.

Second, let $p\to1$. Suppose that $f$ is bounded and locally absolutely continuous, with $f'\in L^1\cap L^\infty$. Theorem~\ref{thm:main}, applied at $p=2$, ensures that $\M f$ is locally absolutely continuous. For any bounded interval $J$ and any $p>1$, H\"older's inequality gives
\[
 \int_J|(\M f)'(x)|\,dx\leq {|J|}^{1-1/p}\norm{(\M f)'}_p\le |J|^{1-1/p}\gamma_p\norm{f'}_p.
\]
As $p\to1$, $\int_J|(\M f)'(x)|\,dx\le\gamma_1\norm{f'}_1$ for all interval $J$. Pick $J=[-R,R]$ and let $R\to\infty$. This yields
\begin{equation}\label{eq:regular-one}
 \norm{(\M f)'}_1=\Var(\M f)\le\norm{f'}_1.
\end{equation}
The variation identity follows from local absolute continuity and integrability of the derivative.

We next pass to $BV$ inputs to recover the full quantitative bound. After replacing $f$ by $|f|$, assume that $f\ge0$ and $|Df|(\R)<\infty$. Let $\rho_\varepsilon$ be a nonnegative smooth compactly supported approximate identity and put $f_\varepsilon=f*\rho_\varepsilon$. A $BV$ representative is bounded, and
\[
 f_\varepsilon'=Df*\rho_\varepsilon\in L^1\cap L^\infty,
 \qquad \norm{f_\varepsilon'}_1\le|Df|(\R).
\]
Thus~\eqref{eq:regular-one} gives $\Var(\M f_\varepsilon)\le|Df|(\R)$. The maximal function $\M f$ is continuous by~\cite[Lemmas 3.3--3.4]{AP}. Fix $x\in\mathbb{R}$. Considering an interval $I$ containing $x$ and taking the supremum over interval $I$ gives
\begin{equation}\label{eq:limsup_leq_Mf}
\begin{split}
&\frac{1}{|I|}\int_If_\epsilon(t)dt=\int_\mathbb{R}\rho_\epsilon(y)\left(\frac{1}{|I|}\int_If(t-y)dt\right)dy\leq\int_\mathbb{R}\rho_\epsilon(y)\M f(x-y)dy\\
&\Rightarrow\M f_\epsilon(x)\leq\underbrace{\big((\M f)*\rho_\epsilon\big)(x)}_{\to\M f(x)}\\
&\Rightarrow\limsup_{\epsilon\to0}\M f_\epsilon(x)\leq\M f(x).
\end{split}
\end{equation}
For the reverse lower-limit inequality, fix any finite averaging interval $I$ containing $x$. Since $f_\varepsilon\to f$ locally in $L^1$, 
\[\frac{1}{|I|}\int_If(t)dt=\liminf_{\epsilon\to0}\frac{1}{|I|}\int_If_\epsilon(t)dt\leq\liminf_{\epsilon\to0}\M f_\epsilon(x)\Rightarrow\M f(x)\leq\liminf_{\epsilon\to0}\M f_\epsilon(x).\]
Together with \eqref{eq:limsup_leq_Mf}, we get $\M f_\epsilon\to\M f$ pointwise everywhere.

By \eqref{eq:variation-lsc} and \eqref{eq:regular-one}, 
\[
 \Var(\M f)\le\liminf_{\varepsilon\downarrow0}\Var(\M f_\varepsilon)
 \le|Df|(\R).
\]
This is the total-variation bound in~\eqref{eq:endpoint-BV}, obtained from Theorem~\ref{thm:main} by $p\downarrow1$ and approximation. The absolute continuity of $\M f$, and therefore the identification of this variation with $\norm{(\M f)'}_1$, is the classical result of Aldaz and P\'erez L\'azaro~\cite[Theorem 2.5]{AP}. We use that regularity result here without reproducing its proof.
\end{proof}

\section*{Acknowledgements} J.M. was partially supported by the Simons Foundation Grant \#453576. 

\section*{Statement on AI use}
The initial ideas and overall structure of this paper are from the authors. Then, in collaboration with AI models (including GPT 6 Astra, Claude Opus 5.5), we developed several technical lemmas, visual illustrations, and concise arguments to complete the proof. All mathematical proofs, derivations, and final conclusions were verified by the authors, who take full responsibility for them.

\end{document}